\documentclass{amsart}
\usepackage[utf8]{inputenc}
\usepackage[T1]{fontenc}

\usepackage{savesym}
\usepackage{hyperref}
\usepackage{amsfonts}
\usepackage{amsmath}
\usepackage{amsthm}
\usepackage{amssymb}
\usepackage{mathrsfs}
\usepackage{amstext}

\theoremstyle{plain}
\newtheorem{Thm}{Theorem}
\newtheorem{lem}{Lemma}

\newtheorem{prop}{Proposition}
\newtheorem{dfn}{Definition}
\newtheorem{rem}{Remark}
\newcommand{\q}{{q(\cdot)}}
\newcommand{\p}{{p(\cdot)}}

\newcommand{\vek}[1]{\boldsymbol{#1}}
\newcommand{\skalpro}[2]{\left\langle\,#1,#2\right\rangle}
\newcommand{\norm}[2]{\left\|\left.{#1}\right|{#2}\right\|}
\newcommand{\schatten}[1]{\mathbb{#1}}
\newcommand{\R}{\schatten{R}}
\newcommand{\C}{\schatten{C}}

\newcommand{\N}{\schatten{N}}

\newcommand{\Sn}{\mathcal{S}(\R^n)}

\newcommand{\SSn}{\mathcal{S}'(\R^n)}
\newcommand{\SSo}{\mathcal{S}'(\Omega)}
\newcommand{\Do}{\mathcal{D}(\Omega)}
\newcommand{\DSo}{\mathcal{D}'(\Omega)}
\newcommand{\Rn}{{\R^n}}
\newcommand{\Bn}{B^{s}_{pq}(\R^n)}
\newcommand{\Fn}{F^{s}_{pq}(\R^n)}

\newcommand{\Bwpxqx}{{B^{\vek{w}}}_{\!\!\!\!\!\!\!p(\cdot),q(\cdot)}(\R^n)}
\newcommand{\Awpxqx}{{A^{\vek{w}}}_{\!\!\!\!\!\!\!p(\cdot),q(\cdot)}(\R^n)}

\newcommand{\Fsxpxqx}{{F^{s(\cdot)}}_{\!\!\!\!\!\!\!\!\!\!\!p(\cdot),q(\cdot)}(\R^n)}

\newcommand{\Bsxpxqx}{{B^{s(\cdot)}}_{\!\!\!\!\!\!\!\!\!\!\!p(\cdot),q(\cdot)}(\R^n)}
\newcommand{\Fwpxqx}{{F^{\vek{w}}}_{\!\!\!\!\!\!\!p(\cdot),q(\cdot)}(\R^n)}
\newcommand{\Fo}{{F^{\vek{w}}}_{\!\!\!\!\!\!\!p(\cdot),q(\cdot)}(\Omega)}
\newcommand{\Bo}{{B^{\vek{w}}}_{\!\!\!\!\!\!\!p(\cdot),q(\cdot)}(\Omega)}
\newcommand{\Ao}{{A^{\vek{w}}}_{\!\!\!\!\!\!\!p(\cdot),q(\cdot)}(\Omega)}

\newcommand{\Plog}{\mathcal{P}^{\log}}

\newcommand{\Konst}{\mathrm{C}}
\newcommand{\mgk}{\mathcal{W}^\alpha_{\alpha_1,\alpha_2}}

\newcommand{\tah}{^\vee}               

\newcommand{\supp}{\operatorname{supp}}

\newcommand{\esssup}{\operatornamewithlimits{ess-sup}}
\newcommand{\essinf}{\operatornamewithlimits{ess-inf}}

\begin{document}
\author{Henning Kempka}
\date{\today}
\title[Intrinsic characterization and the extension operator]{Intrinsic characterization and the extension operator in variable exponent function spaces on special Lipschitz domains}
\begin{abstract}
We study 2-microlocal Besov and Triebel-Lizorkin spaces with variable exponents on special Lipschitz domains $\Omega$. These spaces are as usual defined by restriction of the corresponding spaces on $\Rn$. In this paper we give two intrinsic characterizations of these spaces using local means and the Peetre maximal operator. Further we construct a linear and bounded extension operator following the approach done by Rychkov in \cite{RychExt}, which at the end also turns out to be universal.
\end{abstract}
\maketitle
\section{Introduction}
In this paper we study Besov $\Bo$ and Triebel-Lizorkin spaces $\Fo$ with variable exponents on special Lipschitz domains $\Omega\subset\Rn$, where
\begin{align*}
\Omega=\{(x',x_n)\in\Rn: x_n>\omega(x')\}
\end{align*}
{for a Lipschitz continuous function }$\omega:\R^{n-1}\to\R$.
 Here the variable integrability is defined with measurable functions $\p$ and $\q$ and the variable smoothness is defined in the 2-microlocal sense using admissible weight sequences $\vek{w}=(w_j)_{j\in\N_0}$, see Section \ref{sec:2} for details.\\
Spaces of this type on $\Rn$ have first been considered by Diening, H\"ast\"o and Roudenko in \cite{DHR} by the author in \cite{Kem09}. With also $\q$ variable in the B-case they have been studied by Almeida and H\"ast\"o in \cite{AH} and by the author and Vybiral in \cite{KemV12}.\\
In this paper we obtain intrinsic characterizations of $\Bo$ and $\Fo$ using local means and the Peetre maximal operator. 
Furthermore, a linear and bounded extension operator from the spaces on $\Omega$ to the spaces on $\Rn$ is constructed. 
In the whole work we rely very much on the paper of Rychkov \cite{RychExt} where the same results have been shown for fixed exponents, i.e. $\p=p$, $\q=q$ as constants and $w_j(x)=2^{js}$ with $s\in\R$. Surprisingly, all results remain also true in the variable setting. We refer again to \cite{RychExt} on an exhaustive history of such results.

For variable exponents there are not so many results on intrinsic characterizations and on the extension operator known. An intrinsic characterization for our spaces has been provided in \cite{GonKem} with the help of non smooth atomic characterizations. This approach also works for more general domains than special Lipschitz domains. \\
If $\p=p$ and $\q=q$ are constants, then intrinsic characterizations and an extension operator has been presented by Tyulenev in \cite{Sascha} in the Besov space scale. This work also modified the proofs from Rychkov \cite{RychExt}, but the focus in \cite{Sascha} lies on more general domains and on more general weight sequences where also Muckenhoupt weights are allowed as variable smoothness functions.\\
Further, in \cite{DH} Diening and H\"ast\"o constructed with mollifiers an extension operator for the Sobolev spaces $W^1_\p=F^1_{\p,2}$ from the halfspace to $\Rn$.
  
The paper is structured as follows. We introduce in Section \ref{sec:2} the necessary notation and the Besov and Triebel-Lizorkin spaces $\Bwpxqx$ and $\Fwpxqx$ with variable exponents on $\Rn$. Further, we present there the important local means characterization for these spaces. In Section \ref{sec:3}, we introduce special Lipschitz domains and introduce the spaces $\Bo$ and $\Fo$ as usual by restrictions from the corresponding spaces on $\Rn$. Section \ref{sec:4}  contains the main results of this paper. Here we prove an intrinsic characterization using local means and define a linear and bounded extension operator on $\Bo$ and $\Fo$. This is complemented by Section \ref{sec:5}, where an universal extension operator $\mathcal{E}_u$ is constructed. Here the operator is not depending on the functions $\p,\q$ and the paramters of the weight sequence $\alpha, \alpha_1$ and $\alpha_2$.
\section{Preliminaries}\label{sec:2}
First of all, we introduce all necessary notation. As usual, we denote by  $\Rn$ the n-dimensional Euclidean space, $\N$ denotes the set of natural numbers and we set $\N_0=\N\cup\{0\}$. We write $\eta\approx\xi$ if there exist two constants $c_1,c_2> 0$ with $c_1\eta\leq\xi\leq c_2\eta$.\\
Please be aware that $c>0$ is an universal constant and can change its value from one line to another but is never depending on any variables used in the estimates, except it is clearly noted. The Schwartz space $\Sn$ is the set of all infinitely often differentiable functions on $\Rn$ with rapid decay at infinity. Its topology is generated by the seminorms
\begin{align*}
\|\Phi\|_{k,l}=\sup_{x\in\Rn}(1+|x|)^k\sum_{|\beta|\leq l}|D^\beta\Phi(x)|.
\end{align*}
For a function $\Phi\in\Sn$ we denote by $L_\Phi\in\R$ the number moment conditions the function provides, i.e. $L_\Phi$ is the highest number with
\begin{align}\label{MomentCond}
	\int_\Rn x^\beta \Phi(x)dx=0\quad\text{with }|\beta|<L_\Phi.
\end{align} 
Please note, that for $L_\Phi\leq0$ the function $\Phi$ does not have any moment condition. If not otherwise stated, we define for a function $\Phi\in\Sn$ the dyadic dilates by $\Phi_j(x)=2^{jn}\Phi(2^jx)$ for $j\in\N$ and any $x\in\Rn$. We remark that $\Phi_0$ is not covered by the construction above because it is usually realized with a different function $\Phi_0$ which has different properties compared to $\Phi$.

\subsection{Besov and Triebel-Lizorkin spaces with variable exponents}\label{Sec:SpacesonRn}
Here we introduce the spaces which we are interested in. We study Besov and Triebel-Lizorkin spaces with variable integrability and variable smoothness. We take advantage of the concept of admissible weight sequences to define the variable smoothness.
\begin{dfn}
For fixed real numbers $\alpha\geq0$ and $\alpha_1\leq\alpha_2$ the class of admissible weights $\mgk$ is the collection of all positive weight sequences $\vek{w}=(w_j)_{j\in\N_0}$ on $\Rn$ with:
\begin{enumerate}
	\item There exists a constant $\Konst>0$ such that for fixed $j\in\N_0$ and arbitrary $x,y\in\Rn$
	\begin{align*}
	0<w_j(x)\leq\Konst w_j(y)(1+2^j|x-y|)^\alpha;
	\end{align*}
	\item For any $x\in\Rn$ and any $j\in\N_0$ we have
	\begin{align*}
	2^{\alpha_1}w_j(x)\leq w_{j+1}(x)\leq 2^{\alpha_2}w_j(x).
	\end{align*}
\end{enumerate}
\end{dfn} 
Before introducing the function spaces under consideration we still need to recall some notation. By $\mathcal{S}(\Rn)$ we denote the Schwartz space of all complex-valued rapidly decreasing infinitely differentiable functions on $\Rn$ and by $\mathcal{S}'(\Rn)$ the dual space of all tempered distributions on $\Rn$. For $f\in \mathcal{S}'(\Rn)$ we denote by $\widehat{f}$ the Fourier transform of $f$ and by $f^{\vee}$ the inverse Fourier transform of $f$. \\

Let $\varphi_0\in\mathcal{S}(\Rn)$ be such that
\begin{equation}  \label{phi}
  \varphi_0(x)=1 \quad \mbox{if}\quad |x|\leq 1 \quad \mbox{and} \quad \supp \varphi_0 \subset \{x\in\Rn: |x|\leq 2\}.
\end{equation}
Now define $\varphi(x):=\varphi_0(x)-\varphi_0(2x)$ and set $\varphi_j(x):=\varphi(2^{-j}x)$ for all $j\in\N$. Then the sequence $(\varphi_j)_{j\in\N_0}$ forms a smooth dyadic decomposition of unity, which means
\begin{align*}
\sum_{j=0}^\infty\varphi_j(x)=1\quad\text{for all }x\in\Rn.
\end{align*}
For an open set $\Omega\subset\Rn$ we denote by $\mathcal{P}(\Omega)$ the class of exponents, which are measurable functions $p:\Omega\rightarrow (c,\infty]$ for some $c>0$. Let $p \in \mathcal{P}(\Omega)$, then $p^+:=\esssup_{x\in \Omega}p(x)$ and $p^-:=\essinf_{x\in \Omega}p(x)$. The set $L_{p(\cdot)}(\Omega)$ is the variable exponent Lebesgue space, which consists of all measurable functions $f$ such that for some $\lambda>0$ the modular $\varrho_{{p(\cdot)}}(f/\lambda)$ is finite. The modular is defined by
$$
\displaystyle \varrho_{p(\cdot)}(f):=\int_{\Omega_0 } |f(x)|^{p(x)}\, dx + \text{ess-sup}_{x\in \Omega_{\infty}} |f(x)|.
$$
Here $\Omega_{\infty}$ denotes the subset of $\Omega$ where $p(x)=\infty$ and $\Omega_0= \Omega \setminus \Omega_{\infty}$. The Luxemburg (quasi-)norm of a function $f\in L_{p(\cdot)}(\Omega)$ is given by
$$
\norm{f}{L_{p(\cdot)}(\Omega)}:=\inf\left\{\lambda>0:\varrho_{{p(\cdot)}}\left(\frac{f}{\lambda}\right)\leq 1 \right\}.
$$

In order to define the mixed spaces $\ell_{\q}(L_{\p}(\Omega))$, we need to define another modular. For $p,q \in \mathcal{P}(\Omega)$ and a sequence $(f_{\nu})_{\nu \in \N_0}$ of complex-valued Lebesgue measurable functions on $\Omega$, we define
\begin{align} \label{modular_mixed}
  \varrho_{\ell_{\q}(L_{\p})}(f_{\nu}) = \sum_{\nu=0}^{\infty} \inf\left\{\lambda_{\nu}>0 : \varrho_{\p}\left(\frac{f_{\nu}}{\lambda_{\nu}^{1/\q}} \right)\leq 1 \right\}.
\end{align}
If $q^+ <\infty$, then we can replace \eqref{modular_mixed} by the simpler expression
\begin{align} \label{modular_mixed_norm}
  \varrho_{\ell_{\q}(L_{\p})}(f_{\nu}) = \sum_{\nu=0}^{\infty} \Big\| |f_{\nu}|^{\q} \mid L_{\frac{\p}{\q}}(\Omega) \Big\|.
\end{align}
The (quasi-)norm in the $\ell_{\q}(L_{\p}(\Omega))$ spaces is defined as usual by
\begin{align*}
  \| f_{\nu} \mid \ell_{\q}(L_{\p}(\Omega)) \| = \inf \left\{ \mu>0 : \varrho_{\ell_{\q}(L_{\p})}\left(\frac{f_{\nu}}{\mu}\right) \leq 1\right\}.
\end{align*}

For the sake of completeness, we state also the definition of the space $L_{\p}(\ell_{\q}(\Omega))$. At first, one just takes the norm $\ell_{\q}$ of $(f_{\nu}(x))_{\nu \in \N_0}$ for every $x \in \Omega$ and then the $L_{\p}$-norm with respect to $x \in \Omega$, i.e.
\begin{align*}
  \norm{f_{\nu}}{L_{\p}(\ell_{\q}(\Omega))} = \left\| \left(\sum_{\nu=0}^{\infty} |f_{\nu}(x)|^{q(x)} \right)^{1/q(x)} \mid L_{\p}(\Omega)\right\|. 
\end{align*}
Finally, we also give the definition of smoothness spaces for the exponents. To prove results for the spaces under consideration, like characterizations or the independence of the decomposition of unity, we need this extra regularity for the exponents.
\begin{dfn}
 Let $g\in C(\Omega)$ be a continuous function on $\Omega$.

\begin{enumerate}
\item We say that $g$ is \emph{locally $\log$-H\"older continuous}, abbreviated $g\in C^{\log}_{loc}(\Omega)$,
if there exists $c_{\log}(g)>0$ such that
\begin{align*}
|g(x)-g(y)|\leq\frac{c_{\log}(g)}{\log(e+{1}/{|x-y|})}
\end{align*}
holds for all $x,y\in\Omega$.
\item
We say that $g$ is \emph{globally $\log$-H\"older continuous}, abbreviated $g\in C^{\log}(\Omega)$, if $g$ is
locally $\log$-H\"older continuous and there exists $g_\infty\in\R$ such that
\begin{align*}
|g(x)-g_\infty|\leq\frac{c_{\log}}{\log(e+|x|)}
\end{align*}
holds for all $x\in\Omega$.
\end{enumerate}
\end{dfn}
The logarithmic H\"older regularity classes turned out to be sufficient to have the boundedness of the Hardy-Littlewood maximal operator on $L_\p(\Omega)$ and for further properties we refer to \cite{DHHR} for details. We denote by $p\in\Plog$ any exponent $p\in\mathcal{P}(\Omega)$ with $0<p^-\leq p^+\leq\infty$ and $1/\p\in C^{\log}(\Omega)$.
\begin{rem}
The class $\Plog$ is denoted without underlying class $\Omega$. Having an exponent in $\mathcal{P}(\Rn)$ with $1/p\in C^{\log}(\Rn)$, we can always restrict it to an exponent on $\Omega$. Further by \cite[Proposition 4.1.7]{DHHR} we can always extend an exponent $p\in\mathcal{P}(\Omega)$ with $1/p\in C^{\log}(\Omega)$ to an exponent $\widetilde{p}\in\mathcal{P}(\Rn)$ with $1/\widetilde{p}\in C^{\log}(\Rn)$ without changing the numbers $p^+,p^-,p_\infty$ and $c_{\log}(1/p)$.\\
So, in abuse of notation we always write $p\in\Plog$ and mean either the exponent on $\Rn$ or on $\Omega$, which share in any case the same properties. 
\end{rem}

Now, we are ready to give the definition of the variable exponent spaces which we are interested in.
\begin{dfn}\label{Def:spaces}
 Let $p,q\in\Plog$, $(w_j)_{j\in\N_0}\in\mgk$ and $(\varphi_j)_{j\in\N_0}$ a smooth decomposition of unity. 
 \begin{enumerate}
  \item The variable Besov space $\Bwpxqx$ is the collection of all $f\in\SSn$ with
  \begin{align*}
   \norm{f}{\Bwpxqx}:=\norm{\left(w_j(\cdot)\left(\varphi_j\hat{f}\right)\tah(\cdot)\right)_{j\in\N_0}}{\ell_\q(L_\p(\Rn))}<\infty.
  \end{align*}
  \item For $p^+,q^+<\infty$ the variable Triebel-Lizorkin space $\Fwpxqx$ is the collection of all $f\in\SSn$ with
  \begin{align*}
   \norm{f}{\Fwpxqx}&:=\norm{\left(\sum_{j=0}^\infty|w_j(\cdot)\left(\varphi_j\hat{f}\right)\tah(\cdot)|\q\right)^{1/\q}}{L_\p(\Rn)}\\
   &=\norm{\left(w_j(\cdot)\left(\varphi_j\hat{f}\right)\tah(\cdot)\right)_{j\in\N_0}}{L_\p(\ell_q(\Rn))}.
  \end{align*}
 \end{enumerate}
\end{dfn}
For brevity we write $\Awpxqx$ where either $A=B$ or $A=F$.\\
First definitions of these spaces have been given in \cite{Kem09} and with $\q$ also variable in the Besov case in \cite{KemV12}. Furthermore, there already exist a lot of characterizations of these scales of spaces: namely by local means in \cite{Kem09}, by atoms, molecules and wavelets in \cite{Kem10} and \cite{Drihem}, by ball means of differences in \cite{KemV12} and recently by non-smooth atoms in \cite{KemGon}. If one chooses $0<p,q\leq\infty$ as constants and sets $w_j(x)=2^{js}$ with $s\in\R$ then one recovers the usual Besov and Triebel-Lizorkin spaces $\Bn$ and $\Fn$ studied in great detail in \cite{Tri1}, \cite{Tri2} and \cite{Tri3}.\\
Furthermore, by choosing the weight sequence as $w_j(x)=2^{js(x)}$ with $s\in C_{loc}^{\log}(\Rn)$ we obtain the scales of Besov and Triebel-Lizorkin spaces with variable smoothness and integrability $\Bsxpxqx$ and $\Fsxpxqx$ which have been studied in \cite{DHR} and \cite{AH}.
\subsection{Local means characterization}
Our approach to obtain intrinsic characterizations and an extension operator for $\Bo$ and $\Fo$ for an special Lipschitz domain $\Omega\subset\Rn$ heavily relies on the characterization by local means. To this end, we repeat this characterization for our spaces under consideration from \cite{Kem09} and \cite{KemV12}.
The crucial tool will be the Peetre maximal operator which  assigns to each system
$(\Psi_k)_{k\in\N_0}\subset\Sn$, to each distribution $f\in\SSn$
and to each number $a>0$ the following quantities
\begin{align}\label{_PO_modifiziert}
(\Psi_k^*f)_a(x):=\sup_{y\in\R^n}\frac{|(\Psi_k\ast f)(y)|}{(1+|2^k(y-x)|)^{a}},\quad x\in\R^n
\text{ and }k\in\N_0.
\end{align}

We start with two given functions $\Psi_0,\Psi_1\in\Sn$. We define 
\begin{align*}
\Psi_j(x)=2^{(j-1)n}\Psi_1(2^{(j-1)}x),\quad\text{for $x\in\R^n$ and $j\in\N$.}
\end{align*}
The local means characterization for $\Bwpxqx$ and $\Fwpxqx$ from \cite{KemV12} and \cite{Kem09} then reads.
\begin{prop}\label{Prop:LM}
Let $\vek{w}=(w_k)_{k\in\N_0}\in\mgk$, $p,q\in\Plog$ and let $a>0$, $R\in\N_0$ with $R>\alpha_2$.
Further, let $\Psi_0,\Psi_1$ belong to $\Sn$ with
\begin{align}
\int_\Rn x^\beta \Psi_1(x)dx=0,\quad\text{ for }0\leq|\beta|< R,\label{_LM_MomentCond}
\end{align}
and
\begin{align}
|\hat{\Psi}_0(x)|&>0\quad\text{on}\quad\{x\in\R^n:|x|<\varepsilon\}\label{_LM_Tauber1,5}\\
|\Hat{\Psi}_1(x)|&>0\quad\text{on}\quad\{x\in\R^n:\varepsilon/2<|x|<2\varepsilon\}\label{_LM_Tauber2,5}
\end{align} 
for some $\varepsilon>0$.\\
\begin{enumerate}\item For $a>\frac{n+c_{log}(1/q)}{p^-}+\alpha$ and all $f\in\SSn$ we have
\begin{align*}
\norm{f}{\Bwpxqx}\approx\norm{(\Psi_k\ast f)w_k}{\ell_\q(L_\p(\Rn))}\approx\norm{(\Psi_k^*f)_aw_k}{\ell_\q(L_\p(\Rn))}.
\end{align*}
\item	For $a>\frac{n}{\min(p^-,q^-)}+\alpha$ and all $f\in\SSn$ we have 
\begin{align*}
\norm{f}{\Fwpxqx}\approx\norm{w_k(\Psi_k\ast f)}{L_\p(\ell_\q(\Rn))}\approx\norm{w_k(\Psi_k^*f)_a}{L_\p(\ell_\q(\Rn))}.
\end{align*}
\end{enumerate}
\end{prop}
The local means characterization above easily gives that the norms in Definition \ref{Def:spaces} are independent of the chosen decomposition of unity $(\varphi_j)_{j\in\N_0}$.
\begin{rem}
 \label{rem:phi0zuLM}
\begin{enumerate}
	\item One can rewrite \eqref{_LM_MomentCond} also in $D^\beta\hat{\Psi}_1(0)=0$ for all $|\beta|<R$ or, using our notation, in $L_{\Psi_1}=R$.
	\item Later assertions are done with only one startfunction $\Phi_0\in\Do$ with $\int_\Rn\Phi_0(x)dx\neq0$. From that function one constructs $\Phi(x)=\Phi_0(x)-2^{-n}\Phi(x/2)$ and sets $\Phi_1(x)=2^n\Phi(2x)$.\\ Since $\Phi_0\in\Do\subset \Sn$ is smooth, we can find an $\varepsilon>0$ such that $|\hat{\Phi}_0(x)|>0\quad\text{on}\quad\{x\in\R^n:|x|<\varepsilon\}$ is satisfied. Further, also $\Phi_1\in\Sn$ fulfills $|\hat{\Phi}_1(x)|>0\quad\text{on}\quad\{x\in\R^n:\varepsilon/2<|x|<2\varepsilon\}$ and therefore \eqref{_LM_Tauber1,5} and \eqref{_LM_Tauber2,5} are fulfilled with $\Phi_0$ and $\Phi_1$ instead of the $\Psi_0$ and $\Psi_1$. This also shows, that we can take the functions $\Phi_j(x)=2^{jn}\Phi(2^jx)=2^{(j-1)n}\Phi_1(2^{j-1}x)$ and $\Phi_0$ as basic functions in Proposition \ref{Prop:LM}. 
\end{enumerate}
\end{rem}

\section{Function spaces on special Lipschitz domains}\label{sec:3}
We say that $\Omega\subset \Rn$ with $n\geq2$ is a special Lipschitz domain if it is open and there exists a constant $A>0$ with
\begin{align*}
 \Omega&=\{(x',x_n)\in\Rn: x_n>\omega(x')\}\\
 \intertext{and $\omega:\R^{n-1}\to\R$ is Lipschitz continuous}
 |\omega(x')-\omega(y')|&\leq A |x'-y'|. 
\end{align*}
The function spaces from Section \ref{Sec:SpacesonRn} can be used to define them on domains with the help of Definition \ref{Def:spaces} by restriction.\\
As ususal $\Do=C_0^\infty(\Omega)$ stands for the space of infinitely often differentiable functions with compact support in $\Omega$. Let $\DSo$ be the dual space of distributions on $\Omega$. For $g\in\SSn$ we denote by $g|_\Omega$ its restriction to $\Omega$,
\begin{align*}
 g|_\Omega:\quad(g|\Omega)(\varphi)=g(\varphi)\text{ for all }\varphi\in\Do.
\end{align*}
\begin{dfn}\label{Def:SpacesO}
 Let $\Omega\subset\Rn$ be a special Lipschitz domain as above. Let $p,q\in\Plog$, $(w_j)_{j\in\N_0}\in\mgk$ and $(\varphi_j)_{j\in\N_0}$ a smooth decomposition of unity. 
 \begin{enumerate}
  \item The variable Besov space $\Bo$ on $\Omega$ is the collection of all $f\in\DSo$ such that there exists a $g\in\Bwpxqx$ with $g|_\Omega=f$. Furthermore
  \begin{align*}
   \norm{f}{\Bo}:=\inf\left\{\norm{g}{\Bwpxqx}:g|_\Omega=f\right\}.
  \end{align*}
  \item For $p^+,q^+<\infty$ the variable Triebel-Lizorkin space $\Fo$ on $\Omega$ is the collection of all $f\in\DSo$ such that there exists a $g\in\Fwpxqx$ with $g|_\Omega=f$. Furthermore
  \begin{align*}
   \norm{f}{\Fo}:=\inf\left\{\norm{g}{\Fwpxqx}:g|_\Omega=f\right\}.
  \end{align*}
 \end{enumerate}
\end{dfn}
\begin{rem}
 Usually, one defines function spaces on bounded Lipschitz domains $\Omega$. Then one reduces the proofs and assertions by the localization procedure to special Lipschitz domains. This is done by covering $\partial\Omega$ by finitely many balls $B_j$ and using a decomposition of unity $\Phi_j$ which is adapted to the balls $B_j$. Finally, using pointwise multipliers and rotations (diffeomorphisms) all occurring tasks can be reduced to the case of special Lipschitz domains as described above, see \cite{RychExt} and \cite{Tridomains} for details.\\
 To the best of the authors knowledge there are no results on diffeomorphisms known if the exponents $\p,\q$ are not constant. So we concentrate our studies only on  special Lipschitz domains as above, and leave the case of bounded Lipschitz domains for further research. 
\end{rem}

\section{Intrinsic characterizations and the extension operator}\label{sec:4}
In this section we prove our main results. We give an intrinsic characterization of the spaces from Definition \ref{Def:SpacesO} with the help of an adapted Peetre maximal operator
\begin{align}\label{PeetreO}
 (\Phi_k^*f)^\Omega_a(x):=\sup_{y\in\Omega}\frac{|(\Phi_k\ast f)(y)|}{(1+|2^k(y-x)|)^{a}},\quad x\in\Omega
\text{ and }k\in\N_0.
\end{align}
Here $\Omega\subset\Rn$ with $n\geq2$ is a special Lipschitz domain i.e.
\begin{align*}
 \Omega=\{(x',x_n)\in\Rn: x_n>\omega(x')\}\quad\text{ where}\\
 |\omega(x')-\omega(y')|\leq A|x'-y'|\quad\text{for all $x',y'\in\R^{n-1}$.}
\end{align*}
By $K$ we denote the cone adapted to the special Lipschitz domain with
\begin{align}
 \label{cone}
 K=\{(x',x_n)\in\Rn:|x'|<A^{-1}x_n\}.
\end{align}
This cone has the property that $x+K\in\Omega$ for all $x\in\Omega$ and we denote by $-K=\{-x:x\in K\}$ the reflected cone. The crucial property is now that for all $\gamma\in\mathcal{D}(-K)$ and all $f\in\DSo$ the convolution $(\gamma\ast f)(x)=\skalpro{\gamma(x-\cdot)}{f}$ is well defined in $\Omega$, since $\supp\gamma(x-\cdot)\subset\Omega$ for all $x\in\Omega$.\\
Before coming to the intrinsic characterization and the extension operator we state two useful results which are needed later on. First we need a version of Calderon reproducing formula which was proved in \cite[Proposition 2.1]{RychExt}.
\begin{lem}\label{lem:Calderon}
 Let $\Phi_0\in\mathcal{D}(-K)$ with $\int_\Rn\Phi_0(x)dx\neq0$ be given. Further assume that $\Phi(x)=\Phi_0(x)-2^{-n}\Phi(x/2)$ fulfills
 \begin{align}
  \int_{\Rn}x^\beta\Phi(x)dx=0\qquad\text{for }|\beta|<L_\Phi.
 \end{align}
 Then for any given $L_\Psi\in\R$ there exist functions $\Psi_0,\Psi\in\mathcal{D}(-K)$ with
 \begin{align}
  \int_{\Rn}x^\beta\Psi(x)dx&=0\qquad\text{for }|\beta|<L_\Psi
  \intertext{and for all $f\in\DSo$ we have the identity}
  f&=\sum_{j=0}^\infty\Psi_j\ast\Phi_j\ast f\quad\text{in }\DSo.
 \end{align}
\end{lem}
The second lemma is a Hardy type inequality for the mixed variable spaces. Its proof can be found in \cite[Lemma 9]{KemV12}.
\begin{lem}\label{lem:Hardy}
 Let $p,q\in\mathcal{P}(\Rn)$ and $\delta>0$. For a sequence $(h_j)_{j\in\N_0}$ of measureable functions we denote
 \begin{align*}
  H_l(x)=\sum_{j=0}^\infty2^{-|j-l|\delta}h_j(x).
 \end{align*}
 Then there exist constants $C_1,C_2>0$ depending on $\p,\q$ and $\delta$ with
 \begin{align*}
  \norm{H_l}{\ell_\q(L_\p)}&\leq C_1 \norm{h_l}{\ell_\q(L_\p)}\\
  \norm{H_l}{L_\p(\ell_\q)}&\leq C_2 \norm{h_l}{L_\p(\ell_\q)}.
 \end{align*}
\end{lem}
Now we are ready to formulate our first main theorem about a linear extension operator.
\begin{Thm}\label{Thm:ext}
Let $p,q\in\Plog$ (with $p^+,q^+<\infty$ in the F-case) and $(w_j)\in\mgk$. Further, let $\Phi_0\in\mathcal{D}(-K)$ with $\int\Phi_0(x)dx\neq0$ be given. The function $\Phi(x)=\Phi_0(x)-2^{-n}\Phi(x/2)$ should satisfy $L_\Phi>\alpha_2$.\\
Construct $\Psi_0,\Psi\in\mathcal{D}(-K)$ with $L_\Psi>\frac{n+c_{\log}(1/q)}{\min(p^-,q^-)}+\alpha-\alpha_1$ as in Lemma \ref{lem:Calderon} with
\begin{align*}
f&=\sum_{j=0}^\infty\Psi_j\ast\Phi_j\ast f\quad\text{in }\DSo.
\end{align*}
For any $g:\Omega\to\R$ denote by $g_\Omega$ its extension from $\Omega$ to $\Rn$ by zero. Then the map $\mathcal{E}:\DSo\to\SSn$ with
\begin{align}\label{ExtensionOP}
 f\mapsto\sum_{j=0}^\infty\Psi_j\ast(\Phi_j\ast f)_\Omega
\end{align}
is a linear and bounded extension operator from $\Ao$ to $\Awpxqx$.
\end{Thm}
In more detail, the theorem claims that the series \eqref{ExtensionOP} converges in $\SSn$ for any $f\in\Ao$ to an $\mathcal{E}f$ with: \begin{itemize}
                                                                                                                                        \item $\mathcal{E}f|_\Omega=f$ in the sense of $\DSo$;
                                                                                                                                        \item $\norm{\mathcal{E}f}{\Awpxqx}\leq c\norm{f}{\Ao}$ for any $f\in\Ao$.
                                                                                                                                       \end{itemize}
                                                                                                                                       The theorem above is directly connected to the question of an intrinsic characterization of the spaces $\Ao$, wich will be solved in the next theorem.
                                                                                                                                       \begin{Thm}\label{Thm:Intrinsic}
   Let $\Phi_0\in\mathcal{D}(-K)$ be given with $\int\Phi_0(x)dx\neq0$ and $L_\Phi>\alpha_2$. Further, let $p,q\in\Plog$ and $(w_j)\in\mgk$. For every $f\in\DSo$ we define for $k\in\N_0$                                                                                                                                     
   \begin{align*}
    (\Phi_k^*f)^\Omega_a(x):=\sup_{y\in\Omega}\frac{|(\Phi_k\ast f)(y)|}{(1+|2^k(y-x)|)^{a}},\quad x\in\Omega.
   \end{align*}
\begin{enumerate}
 \item For $a>\frac{n+c_{\log}(1/q)}{p^-}+\alpha$ and any $f\in\DSo$
 \begin{align*}
  \norm{f}{\Bo}\approx\norm{\left(w_k(\Phi_k^*f)^\Omega_a(\cdot)\right)_{k\in\N_0}}{\ell_\q(L_\p(\Omega))}
 \end{align*}
 \item For $a>\frac{n}{\min(p^-,q^-)}+\alpha$, $p^+,q^+<\infty$ and any $f\in\DSo$
 \begin{align}\notag
  \norm{f}{\Fo}&\approx\norm{\left(w_k(\Phi_k^*f)^\Omega_a(\cdot)\right)_{k\in\N_0}}{L_\p(\ell_\q(\Omega))}\\
  &=\norm{\left(\sum_{k=0}^\infty|w_k(\cdot)(\Phi_k^*f)^\Omega_a(\cdot)|^\q\right)^{1/\q}}{L_\p(\Omega)}.\label{fnorm}
 \end{align}
\end{enumerate}															\end{Thm}
\begin{proof}
 The Theorems \ref{Thm:ext} and \ref{Thm:Intrinsic} are so closely connected that they will both be proved in one proof. As usual we restrict to the F-case and outline the necessary modifications for the B-case. By Remark \ref{rem:phi0zuLM} we have the local means characterization from Proposition \ref{Prop:LM} with the functions $\Phi_0$ and $\Phi_j$ constructed from $\Phi_0$.\\
 \underline{First step:} We show $\norm{f}{\Fo}\geq c \norm{\left(\sum_{k=0}^\infty|w_k(\cdot)(\Phi_k^*f)^\Omega_a(\cdot)|^\q\right)^{1/\q}}{L_\p(\Omega)}$. This is an easy consequence of the characterization from Proposition \ref{Prop:LM} using
 \begin{align*}
  (\Phi_k^*f)^\Omega_a(x)\leq(\Phi_k^*g)_a(x) \quad\text{on }\Omega\text{ if }g|_\Omega=f.
 \end{align*}
 \underline{Second step:} We denote the right hand side of \eqref{fnorm} by $\|f\|$. We show if the $\Psi\in\mathcal{D}(-K)$ from Lemma \ref{lem:Calderon} satisfies $L_\Psi>a-\alpha_1$, then for every $f\in\DSo$ with $\|f\|<\infty$ the series in \eqref{ExtensionOP} converges in $\SSn$. Furthermore, the limit $\mathcal{E}f$ satisfies
 \begin{align*}
  \mathcal{E}f|_\Omega=f,\quad\mathcal{E}f\in\Fwpxqx\text{ and }\norm{\mathcal{E}f}{\Fwpxqx}\leq c\|f\|. 
 \end{align*}
 Having this proven, we see that this step actually proves Theorem \ref{Thm:ext} and gives us the $\leq$ estimate in \eqref{fnorm} and therefore finishes the proof of Theorem \ref{Thm:Intrinsic} as well.\\
 \underline{Substep 2.1:} We denote by $X=X^{\vek{w},a}_{\p,\q}$ the space of all sequences $(g^j)_{j\in\N_0}$ of measurable functions $g_j:\Rn\to\C$ with $\|(g^j)\|_X=\norm{\left(\sum_{j=0}^\infty|w_jG^j|^\q\right)^{1/\q}}{L_\p(\Rn)}$, where
 \begin{align*}
  G^j(x)=\sup_{y\in\Rn}\frac{g^j(y)}{(1+2^j|x-y|)^a}.
 \end{align*}
 We claim that if $L_\Psi>a-\alpha_1$, then the series $\sum_{j=0}^\infty\Psi_j\ast g^j$ converges in $\SSn$ and we can find a constant $c>0$ such that for any sequence $(g^j)\in X$
 \begin{align}
  \label{estimate}
  \norm{\sum_{j=0}^\infty\Psi_j\ast g^j}{\Fwpxqx}\leq c\|(g^j)\|_X.
 \end{align}
 To prove \eqref{estimate} we can use the same pointwise estimates as in the proof in \cite{RychExt}. By using the moment conditions on $\Phi$ and $\Psi$ we get using Taylors formula and the compact support of $\Phi$ and $\Psi$
 \begin{align}\notag
  |\Phi_l\ast\Psi_j\ast g^j(x)|&\leq I_{l,j}^aG^j(x)
  \intertext{with}
  I_{j,l}^a=\int_\Rn|(\Phi_l\ast\Psi_j)(z)|(1+2^j|z|)^adz&\leq c\begin{cases}\label{Ilj}
                  2^{(l-j)(L_\Psi-a)},\;& \text{ for }j\geq l\\
                  2^{(j-l)L_\Phi},& \text{ for }j\leq l
                 \end{cases}.
 \end{align}
We use the properties of admissible weight sequences and get
\begin{align*}
 w_l(x)\leq cw_j(x)\begin{cases}
		      2^{-\alpha_1(j-l)},\quad&\text{for }j\geq l\\
		      2^{\alpha_2(l-j)},&\text{for }j\leq l
		    \end{cases}
\end{align*}
and obtain with $\delta=\min(L_\Psi-a+\alpha_1,L_\Phi-\alpha_2)>0$
\begin{align}
 \label{est:1}
 w_l(x)|\Phi_l\ast\Psi_j\ast g^j(x)|\leq cw_j(x)2^{-|j-l|\delta}G^j(x).
\end{align}
Now we use the same arguments as in \cite{RychExt} to finish the proof. If $\|(g^j)\|_X<\infty$, then each $g^j$ is a function of most polynomial growth. Therefore we have $\Psi_j\ast g^j\in\SSn$ and with $\widetilde{w}_l(x)=2^{-l2\delta}w_l(x)$ we obtain from \eqref{est:1}
\begin{align*}
 \norm{\psi_j\ast g^j}{{F^{\widetilde{\vek{w}}}}_{\!\!\!\!\!\!\!p(\cdot),q(\cdot)}(\R^n)}&\leq c \norm{\left(\sum_{l=0}^\infty\left|2^{-2l\delta}2^{-|j-l|\delta}w_j(\cdot)G^j(\cdot)\right|^\q\right)^{1/\q}}{L_\p(\Rn)}\\
 &\leq c\left(\sum_{l=0}^\infty\left|2^{-2l\delta}2^{-|j-l|\delta}\right|^{q^-}\right)^{1/q^-}\norm{w_j(\cdot)G^j(\cdot)}{L_\p(\Rn)}\\
 &\leq c 2^{-j\delta}\norm{w_j(\cdot)G^j(\cdot)}{L_\p(\Rn)}\leq c2^{-j\delta}\|(g^j)\|_X,
\end{align*}
where we used $|l-j|\geq j-l$ and $\ell_{q^-}\hookrightarrow\ell_\q$. Hence, $\sum_{j=0}^\infty\Psi_j\ast g^j$ converges in $\SSn$ due to ${F^{\widetilde{\vek{w}}}}_{\!\!\!\!\!\!\!p(\cdot),q(\cdot)}(\R^n)\subset\SSn$ and we get from \eqref{est:1} the estimate
\begin{align}
 \label{est:2}
 w_l(x)\left|\Phi_l\ast\left(\sum_{j=0}^\infty\Psi_j\ast g^j\right)(x)\right|\leq c\sum_{j=0}^\infty2^{-|j-l|\delta}w_j(x)G^j(x).
\end{align}
Now, using Lemma \ref{lem:Hardy} with $h_j(x)=w_j(x)G^j(x)$ we conclude from \eqref{est:2}
\begin{align}
 \label{est:3}
 \norm{\sum_{j=0}^\infty\Psi_j\ast g^j}{\Fwpxqx}\leq c\|(g^j)\|_X.
\end{align}
\underline{Substep 2.2:}
Finally, we argue as follows to apply our general result \eqref{estimate} to the extension operator from Theorem \ref{Thm:ext}. If $x\in\Omega$, then we have $\sup_{y\in\Omega}\frac{|(\Phi_j\ast f)(y)|}{(1+2^j|x-y|)^a}=(\Phi_j^*f)^\Omega_af(x)$ by definition. If $x\notin\bar{\Omega}$ we can construct a point $\widetilde{x}=(x',2\omega(x')-x_n)\in\Omega$ which is symmetric to $x\notin\bar{\Omega}$ with respect to $\partial\Omega$ in the sense $|\widetilde{x}_n-\omega(x')|=|\omega(x')-x_n|$. Then, by $|\widetilde{x}-y|\leq B|x-y|$ for all $y\in\Omega$, with $B$ depending on the Lipschitz constant $A$, we obtain $\sup_{y\in\Omega}\frac{|(\Phi_j\ast f)(y)|}{(1+2^j|x-y|)^a}\leq c (\Phi_j^*f)^\Omega_af(\widetilde{x})$ for $x\notin\bar{\Omega}$. So we have the estimate
\begin{align*}
 \|(\Phi_j\ast f)_\Omega\|_X\leq c\norm{\left(\sum_{k=0}^\infty|w_k(\cdot)(\Phi_k^*f)^\Omega_a(\cdot)|^\q\right)^{1/\q}}{L_\p(\Omega)}\text{ for all }f\in\DSo.
\end{align*}
Combining this with \eqref{estimate}, we have for all $f\in\DSo$ with $\|f\|<\infty$ that $\mathcal{E}f\in\SSn$ and
\begin{align*}
 \norm{\mathcal{E}f}{\Fwpxqx}\leq c\norm{\left(\sum_{k=0}^\infty|w_k(\cdot)(\Phi_k^*f)^\Omega_a(\cdot)|^\q\right)^{1/\q}}{L_\p(\Omega)}.
\end{align*}
Finally, the supports of $\Psi_0$ and $\Psi$ lie within $-K$ and therefore we obtain using Lemma \ref{lem:Calderon}
\begin{align*}
 \mathcal{E}f|_\Omega=\sum_{j=0}^\infty\Psi_j\ast\Phi_j\ast f=f,
\end{align*}
which completes the proof in the F-case.\\
\underline{Third step:} We can use the same reasoning as above for the B-case. The only difference is in the use of Proposition \ref{Prop:LM}, where the condition on $a>0$ is different in the B-case. This also explains now the condition on $L_\Psi$ in Theorem \ref{Thm:ext}, where we have just taken a maximal value for $a>0$.
\end{proof}
It is also possible to get an intrinsic characterization of $\Ao$ by using just the convolutions $\Phi_j\ast f$ instead of the maximal functions $(\Phi_j^*f)^\Omega_a$ as in the local means characterization in Proposition \ref{Prop:LM}.\\
To that end, we introduce the space $\SSo$ as subspace of $\DSo$ by restriction as
\begin{align*}
 \SSo:=\{f\in\DSo&:\exists c_f,M_f>0 \text{ with }|\skalpro{f}{\gamma}|\leq c_f\|\gamma\|_{M_f}\forall\gamma\in\Do\}
 \intertext{where}
 \|\gamma\|_{M_f}&=\sup_{y\in\Omega,|\beta|\leq M_f}|D ^\beta \gamma(y)|(1+|y|)^{M_f}.
\end{align*}
From \cite[Proposition 3.1]{RychExt} we have the following characterization of this class.
\begin{prop}
 We have $f\in\SSo$ if and only if there exists a $g\in\SSn$ such that $g|_\Omega=f$.
\end{prop}
\begin{rem}
Since all appearing function spaces $\Ao$ are also defined by restriction we have $\Ao\subset\SSo$. Therefore, the proposition above shows that it is no restriction to use $f\in\SSo$ instead of $f\in\DSo$. 
\end{rem}

Furthermore, we also need another lemma which can be seen as the replacement for the boundedness of the Hardy-Littlewood maximal operator which is of no use in our variable exponent spaces. We refer to \cite{DHR} and \cite{AH} for the proofs of this lemma.
\begin{lem}\label{lem:Maxop}
Let $p,q\in\Plog$ and $\eta_{\nu,m}(x)=2^{n\nu}(1+2^\nu|x|)^{-m}$.
\begin{enumerate}
 \item If $p^-\geq1$ and $m>n+c_{\log}(1/q)$, then there exists a constant $c>0$ such that for all sequences $(f_\nu)_{\nu\in\N_0}\in\ell_\q(L_\p(\Rn))$
 \begin{align*}
  \norm{f_\nu\ast\eta_{\nu,m}}{\ell_\q(L_\p(\Rn))}\leq c\norm{f_\nu}{\ell_\q(L_\p(\Rn))}.
 \end{align*}
  \item If $1<p^-\leq p^+<\infty$ and $1<q^-\leq q^+<\infty$ and $m>n$, then there exists a constant $c>0$ such that for all sequences $(f_\nu)_{\nu\in\N_0}\in L_\p(\ell_\q(\Rn))$
 \begin{align*}
  \norm{f_\nu\ast\eta_{\nu,m}}{L_\p(\ell_\q(\Rn))}\leq c\norm{f_\nu}{L_\p(\ell_\q(\Rn))}.
 \end{align*}
\end{enumerate}

\end{lem}

Now, the local means intrinsic characterization for the spaces $\Ao$ reads as follows.
\begin{Thm}\label{Thm:LMo}
Let $\Phi_0\in\mathcal{D}(-K)$ be given with $\int\Phi_0(x)dx\neq0$ and $L_\Phi>\alpha_2$. Further, let $p,q\in\Plog$ and $(w_j)\in\mgk$. 
\begin{enumerate}
 \item For all $f\in\SSo$ we have
 \begin{align*}
  \norm{f}{\Bo}\approx\norm{\left(w_k(\Phi_k\ast f)(\cdot)\right)_{k\in\N_0}}{\ell_\q(L_\p(\Omega))}
 \end{align*}
 \item For $p^+,q^+<\infty$ and all $f\in\SSo$ we have
 \begin{align*}
  \norm{f}{\Fo}&\approx\norm{\left(\sum_{k=0}^\infty|w_k(\cdot)(\Phi_k\ast f)(\cdot)|^\q\right)^{1/\q}}{L_\p(\Omega)}.
 \end{align*}
\end{enumerate}		
 
\end{Thm}
\begin{proof}
Clearly, we want to take the intrinsic norm given in Theorem \ref{Thm:Intrinsic} as a starting point. To use this characterization we need $L_\Phi>\alpha_2$ and choose suitable functions $\Psi_0,\Psi$ which fulfill \eqref{lem:Calderon} with $L_\Psi>a-\alpha_1$. Furthermore, we take the $a>0$ as large as needed in Theorem \ref{Thm:Intrinsic}.\\
\underline{First step:} The $\geq$ inequality follws easily by observing $(\Phi_k^* f)_a^\Omega(x)\geq (\Phi_k\ast f)(x)$.\\
\underline{Second step:} One way to prove the $\leq$ inequality  would be to consult the proof of \cite[Theorem 13]{KemV12} and to modify everything from $\Rn$ to $\Omega$. Instead we use formula (3.4) in \cite{RychExt}
 \begin{align}
  \label{est:3.4}
  |(\Phi_j\ast f)(x)|^r\leq c\sum_{k=j}^\infty2^{(j-k)L_\Psi r}2^{kn}\int_\Omega\frac{|(\Phi_k\ast f)(y)|^r}{(1+2^j|x-y|)^{ar}}dy
 \end{align}
 which was obtained by pointwise manipulations only. Here $r>0$ and the constant $c>0$ is independent of $f\in\SSo$, $x\in\Omega$ and $j\in\N_0$.\\
 Now, dividing \eqref{est:3.4} by $(1+2^j|x-z|)^{ar}$ and using on the left hand side $1+2^j|y-z|\leq(1+2^j|x-z|)(1+2^j|x-y|)$ gives us by taking the supremum with respect to $x\in\Omega$
 \begin{align*}
  \left((\Phi_j^*f)^\Omega_a(z)\right)^r\leq c\sum_{k=j}^\infty2^{(j-k)L_\Psi r}2^{kn}\int_\Omega\frac{|(\Phi_k\ast f)(y)|^r}{(1+2^j|y-z|)^{ar}}dy
 \end{align*}
 We multiply with $w_j(z)^r$ and use the estimates $(1+2^k|y-z|)^{ar}\leq2^{(k-j)ar}(1+2^j|y-z|)^{ar}$ and $w_j(z)\leq \Konst 2^{(j-k)\alpha_1}w_k(y)(1+2^k|y-z|)^\alpha$ and obtain
 \begin{align}
  \notag
  \left(w_j(z)(\Phi_j^*f)^\Omega_a(z)\right)^r&\leq c\sum_{k=j}^\infty2^{(j-k)(L_\Psi-a+\alpha_1) r}2^{kn}\int_\Omega\frac{w_k^r(y)|(\Phi_k\ast f)(y)|^r}{(1+2^k|y-z|)^{(a-\alpha)r}}dy
  \intertext{which can be rewritten with $\delta=L_\Psi-a+\alpha_1>0$ in}
    \left(\chi_\Omega(z)w_j(z)(\Phi_j^*f)^\Omega_a(z)\right)^r&\leq c\sum_{k=j}^\infty2^{(j-k)\delta r}\left[\left(\chi_\Omega w_k(\Phi_k\ast f)\right)^r\ast\eta_{k,(a-\alpha)r}\right](z).\label{est:4}
 \end{align}
 Now, we use the usual procedure to end the proof. In the F-case we choose $r>0$ with $\frac{n}{a-\alpha}<r<\min(p^-,q^-)$. This is possible due to the conditions of the theorem and we get $p/r,q/r\in\Plog$ with $1<p^-/r\leq p^+/r<\infty$, $1<q^-/r\leq q^+/r<\infty$. Applying the $L_{\p/r}(\ell_{\q/r}(\Rn))$ norm on \eqref{est:4} we conclude by using Lemmas \ref{lem:Hardy} and \ref{lem:Maxop} 
  \begin{align*}
   \norm{w_j(z)(\Phi_j^*f)^\Omega_a(z)}{L_\p(\ell_{\q}(\Omega))}^r&=\norm{\left(\chi_\Omega(z)w_j(z)(\Phi_j^*f)^\Omega_a(z)\right)^r}{L_{\p/r}(\ell_{\q/r}(\Rn))}\\
   &\leq c\norm{\left(\chi_\Omega w_k(\Phi_k\ast f)\right)^r\ast\eta_{k,(a-\alpha)r}}{L_{\p/r}(\ell_{\q/r}(\Rn))}\\
   &\leq c\norm{\left(\chi_\Omega w_k(\Phi_k\ast f)\right)^r}{L_{\p/r}(\ell_{\q/r}(\Rn))}\\
   &= c\norm{ w_k(\Phi_k\ast f)}{L_{\p}(\ell_{\q}(\Omega))}^r.
  \end{align*}
  This finishes the proof in the F-case using Theorem \ref{Thm:Intrinsic}. In the B-case the same reasoning by taking the $\ell_{\q/r}(L_{\p/r}(\Rn))$ norm of \eqref{est:4} works; only the parameter $r>0$ has to be chosen as
  \begin{align*}
   \frac{n+c_{\log}(1/q)}{a-\alpha}<r<p^-.
   \end{align*}
\end{proof}

\section{A universal extension operator}\label{sec:5}
The extension operator $\mathcal{E}$ from Theorem \ref{Thm:ext} has the serious drawback that it only works for special values of $\p,\q$ and $\alpha_1,\alpha_2,\alpha$. This is due to the fact that all conditions depend on the number of moments we have for the functions $\Phi$ and $\Psi$. More precisely, we know that for fixed numbers of moments $L_\Phi,L_\Psi$ the extension operator works for 
\begin{align*}
 L_\Phi>\alpha_2\qquad\text{and}\qquad L_\Psi>\frac{n+c_{\log}(1/q)}{\min(p^-,q^-)}+\alpha-\alpha_1.
\end{align*}
A good try to widen this region would be to choose $\Phi,\Psi\in\mathcal{D}(-K)$ with $L_\Psi=L_\Phi=\infty$, but clearly this is impossible. Fortunately, this can be done if $\Phi,\Psi\in\mathcal{S}(\Rn)$ which are not compactly supported in $-K$, but have support in $-K$ and rapid decay at infinity.
\begin{Thm}
 \label{Thm:Universal}
 \begin{enumerate}
  \item\label{romeins} There exist functions $\Phi_0,\Phi,\Psi_0,\Psi\in\Sn$ with supports in $K=\{(x',x_n)\in\Rn:|x'|<A^{-1}x_n\}$ with $L_\Psi=L_\Phi=\infty$ and 
  \begin{align*}
   f=\sum_{k=0}^\infty\Psi_k\ast\Phi_k\ast f\qquad\text{holds for all }f\in\SSo.
  \end{align*}
  \item The map $\mathcal{E}_u:\SSo\to\SSn$ defined with the functions from \ref{romeins} by
  \begin{align*}
   f\mapsto\sum_{k=0}^\infty\Psi_k\ast(\Phi_k\ast f)_\Omega
  \end{align*}
  yields a linear bounded extension operator from $\Ao$ to $\Awpxqx$ for all admissible exponents $\p,\q$ and $(w_j)\in\mgk$.

 \end{enumerate}

\end{Thm}
The proof of this theorem can be copied word by word from the proof of \cite[Theorem 4.1]{RychExt}. The crucial part there is to construct the needed functions $\Phi_0,\Phi,\Psi_0,\Psi\in\Sn$ with supports in $K=\{(x',x_n)\in\Rn:|x'|<A^{-1}x_n\}$ with $L_\Psi=L_\Phi=\infty$ which consists in a modification of Stein's function\cite[§ VI.3]{Stein}. Finally, with that functions satisfying Calderon's reproducing formula one has to revisit the proof of Theorem \ref{Thm:ext}. Actually, there is only one difficulty to overcome: we estimated in \eqref{Ilj} by using the compact support of the functions $\Phi_0,\Phi,\Psi_0,\Psi\in\Sn$. Since we do not have any compact support of these functions anymore we have to use \cite[Lemma 2.1]{BuiPT} and  the same estimate \eqref{Ilj} can be achieved.

\end{document}